\documentclass[11pt]{article}

\usepackage[T1]{fontenc}
\usepackage[utf8]{inputenc}
\usepackage{lmodern}

\usepackage{amsmath,amssymb,amsthm,bm}

\usepackage[a4paper,margin=1.15in]{geometry}
\usepackage{graphicx}
\usepackage{caption}
\usepackage{booktabs,array}
\usepackage{float}
\usepackage{framed}
\usepackage[colorlinks=true,
linkcolor=blue,
citecolor=blue,
urlcolor=blue]{hyperref}

\usepackage{titlesec}
\titleformat{\section}
{\normalfont\Large\bfseries}{\thesection.}{0.75em}{}
\titleformat{\subsection}
{\normalfont\large\bfseries}{\thesubsection.}{0.75em}{}
\titleformat{\subsubsection}
{\normalfont\normalsize\bfseries}{\thesubsubsection.}{0.75em}{}

\usepackage{abstract}

\newtheorem{theorem}{Theorem}[section]
\newtheorem{proposition}[theorem]{Proposition}
\newtheorem{lemma}[theorem]{Lemma}

\newtheorem{remark}[theorem]{Remark}

\newcommand{\keywords}[1]{%
  \par\noindent\textbf{Keywords: }#1
}

\newcommand{\subjclass}[2][]{%
  \par\noindent\textbf{MSC 2020: }#2
}

\title{\vspace{-1em}
	Adaptive estimation of Sobolev-type energy functionals on the sphere
	\vspace{-0.5em}}

\author{
	Claudio Durastanti\\
	\small Department of Basic and Applied Sciences for Engineering\\
	\small Sapienza University of Rome\\
	\small\texttt{claudio.durastanti@uniroma1.it}
}

\date{\small \today}

\date{\small \today}
\begin{document}
	\maketitle

	\begin{abstract}
We study the estimation of quadratic Sobolev-type integral functionals of an
unknown density on the unit sphere.
The functional is defined through fractional powers of the Laplace--Beltrami
operator and provides a global measure of smoothness and spectral energy.
Our approach relies on spherical needlet frames, which yield a localized
multiscale decomposition while preserving tight frame properties in the
natural square-integrable function space on the sphere.

We construct unbiased estimators of suitably truncated versions of the
functional and derive sharp oracle risk bounds through an explicit
bias--variance analysis.
When the smoothness of the density is unknown, we propose a Lepski-type
data-driven selection of the resolution level.
The resulting adaptive estimator achieves minimax-optimal rates over Sobolev
classes, without resorting to nonlinear or sparsity-based methods.
\end{abstract}



\section{Introduction}

Quadratic integral functionals of an unknown density play a central role in
nonparametric statistics.
They arise naturally in problems of goodness-of-fit testing, estimation of
energy and smoothness, signal detection, and the analysis of inverse and
ill-posed problems.
Classical examples include the $L^2$-norm of a density, the integrated squared
gradient, and more general Sobolev-type seminorms.
The estimation of such functionals from i.i.d.\ observations has a long history (see, among others,
\cite{BickelRitov1988,GineNickl2008,Ibragimov1985,Laurent1996}).

In this paper we study the estimation of quadratic Sobolev-type functionals of an
unknown density defined on the unit sphere.
Specifically, we consider functionals of the form
\[
T_r(f)
=
\left\| (-\Delta_{\mathbb S^d})^{r/2} f \right\|_{L^2(\mathbb S^d)}^2,
\qquad r \ge 0,
\]
where $\Delta_{\mathbb S^d}$ denotes the Laplace--Beltrami operator.
Such functionals provide a global measure of smoothness and spectral energy of
the density and are naturally adapted to the geometry of the sphere.
They arise in the analysis of directional and spherical data, in regularization
theory on manifolds, and in applications where rotational invariance and
geometric structure are essential (see, e.g.,
\cite{AtkinsonHan2012,MarinucciPeccati2011}).

The Sobolev-type functional $T_r(f)$ admits a natural interpretation for small
values of $r$.
When $r=0$, $T_0(f)=\|f\|_{L^2(\mathbb S^d)}^2$ coincides with the classical
quadratic energy of the density and represents the simplest example of a
quadratic functional.
For $r=1$, integration by parts shows that
$T_1(f)=\|\nabla_{\mathbb S^d} f\|_{L^2(\mathbb S^d)}^2$, which quantifies the
global oscillatory behavior or roughness of the density.
The case $r=2$ corresponds to the squared $L^2$-norm of the Laplace--Beltrami
operator applied to $f$ and measures second-order variation and curvature.
More generally, fractional values of $r$ interpolate smoothly between these
regimes and provide a spectral notion of smoothness that is particularly natural
on compact manifolds such as the sphere.

From a statistical perspective, increasing values of $r$ place progressively
more weight on high-frequency components of the density.
As a consequence, the estimation of $T_r(f)$ becomes increasingly sensitive to
noise, making truncation and multiscale regularization intrinsic to the problem.
This observation motivates the use of needlet decompositions and truncated
estimators throughout the paper.

The spherical setting introduces specific analytic and statistical challenges
compared to the Euclidean case.
Classical notions of partial derivatives and Fourier analysis are not globally
available, and must be replaced by spectral decompositions associated with the
Laplace--Beltrami operator.
From a statistical perspective, the sphere is the natural state space for
directional data, where rotational invariance and the absence of global
coordinates preclude direct use of Euclidean techniques.
Moreover, estimation of the Sobolev functional $T_r(f)$ is intrinsically
ill-posed due to the amplification of high-frequency components, making
truncation or regularization an essential part of the problem.
Analytically, the sphere provides a canonical compact Riemannian manifold with a
discrete spectrum, allowing Sobolev regularity and quadratic energy functionals
to be defined in a natural and intrinsic way.
While several elements of the analysis extend to more general compact manifolds
with suitable multiscale frames, focusing on the spherical setting allows us to
isolate the core geometric and statistical features without additional technical
overhead (see, e.g.,
\cite{BickelRitov1988,GineNickl2008,Ibragimov1985}).

Our approach is based on spherical needlet frames, which provide a localized
multiscale representation of functions on the sphere while retaining tight frame
properties in $L^2(\mathbb S^d)$.
Needlets combine the global spectral structure of spherical harmonics with sharp
spatial localization and have proven effective in a wide range of statistical
tasks on the sphere; see, among others,
\cite{BaldiKerkyacharianMarinucciPicard2009,Durastanti2016,
	MarinucciPeccati2011,NarcowichPetrushevWard2006}.
In the present context, the needlet decomposition allows us to reorganize the
Sobolev energy of the density across dyadic frequency bands and to construct
natural truncated versions of the target functional.

We first introduce unbiased quadratic estimators of suitably truncated Sobolev
functionals and derive sharp oracle risk bounds through an explicit
bias--variance analysis.
When the smoothness of the underlying density is known, these estimators attain
minimax-optimal rates over Sobolev classes.
The resulting oracle bounds clarify the role of the resolution level in
balancing truncation bias and stochastic variability and serve as a benchmark
for adaptive procedures (see, for example,
\cite{GineNickl2008,Laurent1996,Tsybakov2009}).

A central contribution of the paper is the construction of an adaptive estimator
that does not rely on prior knowledge of the smoothness parameter.

Unlike existing needlet-based methods primarily concerned with density estimation
or local feature recovery, the present work focuses on global Sobolev-type energy
functionals and their adaptive estimation, requiring a fundamentally different
bias--variance and truncation analysis.

Exploiting the monotone behavior of the bias and variance terms induced by the
multiscale needlet decomposition, we select the resolution level in a
data-driven fashion using a Lepski-type procedure
\cite{GoldenshlugerLepski2011,Lepski1991}.
We show that the resulting adaptive estimator achieves the same rates of
convergence as the oracle estimator, uniformly over a range of Sobolev
regularities.
Importantly, adaptivity is achieved through the selection of a single global
resolution parameter and does not require nonlinear or sparsity-based
techniques, reflecting the intrinsically quadratic and Hilbertian nature of the
Sobolev functional.% in contrast with Besov-adaptive approaches that rely on local sparsity and coefficient-wise thresholding.

The paper is organized as follows.
Section~\ref{sec:background} introduces the analytic framework, recalling Sobolev geometry on the sphere and its multiscale characterization via spherical needlets.
Section~\ref{sec:estimation_strategy} develops the statistical methodology, including truncation of the Sobolev functional, construction of unbiased needlet-based estimators, and oracle bias--variance bounds.
Section~\ref{sec:adaptivity} addresses adaptive estimation through a Lepski-type selection of the resolution level and establishes minimax-optimal adaptive rates over Sobolev classes.
Section~\ref{sec:numerics} provides numerical illustrations of the theoretical bias--variance trade-off and resolution selection.
All proofs are collected in Section~\ref{sec:proofs}.

\section{Sobolev geometry and multiscale representations on the sphere}\label{sec:background}
This section introduces the analytic framework underlying the estimation
procedures developed in the paper.
We recall basic elements of harmonic analysis on the sphere and describe how
Sobolev regularity can be characterized through spectral decompositions
associated with the Laplace--Beltrami operator.
We then introduce spherical needlet frames, which provide a multiscale and
spatially localized representation while retaining tight frame properties in
$L^2(\mathbb S^d)$.
The combination of spectral Sobolev geometry and needlet decompositions allows
us to reorganize Sobolev-type energies across dyadic frequency bands, a feature
that is crucial for truncation, bias control, and adaptive estimation.
All the material presented in this section is standard and can be found, in
various forms, in the literature on harmonic analysis and multiscale methods on
the sphere (cf., for instance,
\cite{AtkinsonHan2012,MarinucciPeccati2011,NarcowichPetrushevWard2006}).

Throughout this paper, for two nonnegative quantities $a$ and $b$, we write $a \lesssim b$ if there exists
a constant $C>0$, independent of the sample size $n$ and of the resolution level,
such that $a \le C\, b$.
Similarly, $a \gtrsim b$ means $b \lesssim a$, and $a \asymp b$ indicates that both
$a \lesssim b$ and $a \gtrsim b$ hold.
Unless otherwise stated, all implicit constants may depend on fixed structural
parameters such as the dimension $d$, the Sobolev order $r$, and the choice of the
needlet window, but never on $n$ or on the resolution level.
\medskip

\subsection{Spectral analysis and Sobolev spaces on the sphere}

Let $\Delta_{\mathbb S^d}$ denote the Laplace--Beltrami operator on $\mathbb S^d$.
Its eigenfunctions are the spherical harmonics
\[
\left\{ Y_{\ell,m} : \ell \ge 0,\ m = 1,\ldots,M_{\ell;d} \right\},
\]
which form a complete orthonormal basis of $L^2(\mathbb S^d)$.
The corresponding eigenvalues are given by
\[
e_{\ell,d} = \ell(\ell + d - 1), \qquad \ell \ge 0.
\]
Note that $e_{\ell,d} \asymp \ell^2$ as $\ell \to \infty$.
For each fixed $\ell$, the multiplicity $M_{\ell;d}$ denotes the dimension of
the eigenspace associated with $e_{\ell,d}$ and is given explicitly by
\[
M_{\ell;d}
=
\frac{(2\ell+d-1)(\ell+d-2)!}{\ell!\,(d-1)!},
\qquad \ell \ge 0;
\]
see, for instance, \cite{AtkinsonHan2012}.
In particular, $M_{\ell;d} \asymp \ell^{d-1}$ as $\ell \to \infty$.

Any function $f \in L^2(\mathbb S^d)$ admits the spherical harmonic expansion
\[
f(x)
=
\sum_{\ell=0}^\infty
\sum_{m=1}^{M_{\ell;d}}
a_{\ell,m}\, Y_{\ell,m}(x),
\qquad
a_{\ell,m}
=
\langle f, Y_{\ell,m} \rangle_{L^2(\mathbb S^d)}.
\]

\medskip

The spherical harmonic representation provides a direct link between the
observed sample and the target functional.
Indeed, for any $(\ell,m)$,
\[
a_{\ell,m}
=
\mathbb E_f\!\left[ Y_{\ell,m}(X_1) \right],
\]
so that the harmonic coefficients of $f$ can be estimated from the data through
their empirical counterparts
\[
\hat a_{\ell,m}
=
\frac{1}{n} \sum_{i=1}^n Y_{\ell,m}(X_i).
\]
As a consequence, quadratic functionals of the form introduced below can be
naturally estimated from the observations via their spectral representation.

\begin{remark}[Harmonic versus multiscale representations]
\label{rem:harmonic_vs_needlet}
Although the Sobolev functional admits a direct spectral representation in terms of the harmonic coefficients $a_{\ell,m}$, estimators based on hard truncation of the harmonic expansion lead to a rigid bias--variance structure and are poorly suited for adaptive resolution selection.
In particular, truncation in the harmonic domain does not induce a natural monotone ordering of estimators across scales, making the implementation and analysis of data-driven selection rules delicate.
\end{remark}
\medskip

To introduce the notion of smoothness underlying the functional of interest,
we now recall the definition of Sobolev spaces on the sphere, where classical notions of partial derivatives are not available in
global coordinates.
Sobolev spaces are therefore naturally defined through the spectral properties
of the Laplace--Beltrami operator, which is the canonical elliptic operator
associated with the Riemannian structure of $\mathbb S^d$.
In this framework, fractional powers of $(-\Delta_{\mathbb S^d})$ define
derivatives in the spectral sense, in full analogy with Fourier multipliers in
Euclidean spaces.
This spectral definition is known to be equivalent to other analytic
characterizations of Sobolev spaces on the sphere, including formulations based
on reproducing kernels, square functions, and integral representations; see,
for instance,
\cite{AtkinsonHan2012,BarceloLuquePerezEsteva2020,
	BrauchartDick2013,Hebey1996,MarinucciPeccati2011}.
Under these equivalent definitions, the Sobolev space $H^r(\mathbb S^d)$
consists of functions whose spherical harmonic coefficients satisfy suitable
weighted $\ell^2$ summability conditions.

\medskip

For $r \ge 0$, here and throughout, $f^{(r)}$ denotes the $r$-th order derivative
of $f$ in the spectral sense induced by the Laplace--Beltrami operator, namely,  
\begin{equation}\label{eqn:def}
	f^{(r)}(x)=(-\Delta_{\mathbb S^d})^{r/2} f (x) \qquad x\in \mathbb{S}^d.
\end{equation}
Accordingly, $f^{(r)}$ admits the harmonic expansion:
\begin{equation}\label{eqn:harmexpder}
	f^{(r)} (x)  = \sum_{\ell=0}^\infty
	\sum_{m=1}^{M_{\ell;d}}
	a_{\ell,m}^{(r)}\, Y_{\ell,m}(x), \qquad x \in \mathbb{S}^d
\end{equation}
where, for any $\ell \geq 0$, $m=1,\ldots,M_{\ell;d}$, 
\begin{equation}\label{eqn:harmder}
	a^{(r)}_{\ell,m}
	=
	e_{\ell,d}^{\,r/2}\, a_{\ell,m}
\end{equation}
denotes the corresponding spherical harmonic coefficient.

\medskip

\subsection{Sobolev-type quadratic functionals}

The object of interest in this paper is the quadratic Sobolev-type functional
\begin{equation}\label{eqn:funcdef}
	T_r(f)
	=
	\left\| (-\Delta_{\mathbb S^d})^{r/2} f \right\|_{L^2(\mathbb S^d)}^2
	=
	\int_{\mathbb S^d}
	\left|  f^{(r)}(x) \right|^2
	\, \mathrm{d}x .
\end{equation}
We emphasize that $T_r(f)$ corresponds to the Sobolev seminorm of order $r>0$, as constant
functions belong to the kernel of $(-\Delta_{\mathbb S^d})^{r/2}$.
It is
finite if and only if $f$ belongs to the Sobolev space $H^r(\mathbb S^d)$. Using Equations \eqref{eqn:harmexpder}, \eqref{eqn:harmder}, and \eqref{eqn:funcdef}, $T_r(f)$
admits the exact representation
\begin{equation}\label{eqn:harmexp}
	\begin{split}
	T_r(f)
	& =
	\sum_{\ell=0}^\infty
	\sum_{m=1}^{M_{\ell;d}}
	\left| a^{(r)}_{\ell,m} \right|^2\\
	&=\sum_{\ell=0}^\infty
	\sum_{m=1}^{M_{\ell;d}}
	e_{\ell,d}^{r}\left| a_{\ell,m} \right|^2  .
	\end{split}
\end{equation}

\medskip

Quadratic functional estimation has a long history in nonparametric statistics;
see, among others,
\cite{BickelRitov1988,GineNickl2008,Ibragimov1985,Laurent1996}.
The functional $T_r(f)$ provides a quantitative measure of the smoothness and
spectral energy of the density $f$ on the sphere.
Low values of $r$ correspond to global energy measures, while larger values of
$r$ increasingly emphasize high-frequency oscillations and fine-scale
irregularities of $f$.

From a statistical perspective, the order $r$ controls the degree of ill-posedness of the estimation problem.
Larger values of $r$ correspond to stronger amplification of high-frequency components through the Sobolev weights $e_{\ell,d}^r$, leading to faster variance growth of empirical estimators.
As a consequence, truncation of the spectral or multiscale expansion becomes increasingly critical as $r$ increases.
This feature will be reflected explicitly in the bias--variance tradeoff of truncated estimators developed in Section~\ref{sec:estimation_strategy}.

Throughout the paper, we assume that the density $f$ belongs to $H^s(\mathbb S^d)$ for some $s>r$ and is essentially bounded.
In particular, the condition $s>d/2$ ensures that $f\in L^\infty(\mathbb S^d)$ by Sobolev embedding, which is sufficient for the variance bounds derived below.
When $s\le d/2$, boundedness may be imposed directly as an additional assumption.

\medskip

\subsection{Needlet frame and functional decomposition}\label{sec:needlet}

Needlets provide a localized multiscale representation of functions on the
sphere, combining the spectral properties of spherical harmonics with spatial
localization.
Unlike the global harmonic basis, needlets are simultaneously localized in
frequency and in space, a feature that is particularly advantageous in
nonparametric statistical applications.
Needlet-based methods have been extensively used for density estimation,
regression, thresholding, testing, and the analysis of spherical random fields;
see, for instance,
\cite{
	BaldiKerkyacharianMarinucciPicard2009,
	Durastanti2016,
	DurastantiGellerMarinucci2012,
	DurastantiShevchenko2026,
	MarinucciPeccati2011,
	Monnier2011,
	NarcowichPetrushevWard2006b,
	NarcowichPetrushevWard2006}.

\medskip

Fix $B>1$ and let $b \in C^\infty(\mathbb R)$ be a nonnegative function supported
on $[B^{-1},B]$ such that
\[
\sum_{j\ge 0} b^2\!\left(\frac{\ell}{B^j}\right)=1,
\qquad \text{for all } \ell\ge1.
\]
The index $j\ge0$ denotes the resolution level, corresponding to a dyadic
frequency band of width $B^j$.
For each fixed $j$, let
$\{(\xi_{j,k},\lambda_{j,k}) : k=1,\dots,K_j\}$
be a cubature rule on $\mathbb S^d$ exact for spherical polynomials of degree up
to order $B^{j+1}$, where the index $k$ labels spatial locations on the sphere
and $K_j\asymp B^{dj}$.

The needlet system is defined by
\[
\psi_{j,k}(x)
=
\sqrt{\lambda_{j,k}}
\sum_{\ell\in \Lambda_j}
b\!\left(\frac{\ell}{B^j}\right)
\sum_{m=1}^{M_{\ell;d}}
Y_{\ell,m}(\xi_{j,k})\, Y_{\ell,m}(x),
\]
where $\Lambda_j = \{ \ell : B^{j-1} \le \ell \le B^{j+1} \}$.
The collection $\{\psi_{j,k}:j \geq 0; k=1,\ldots,K_j\}$ forms a tight frame of $L^2(\mathbb S^d)$, namely
\[
\|g\|_{L^2(\mathbb S^d)}^2
=
\sum_{j\geq 0}\sum_{k=1}^{K_j} |\langle g,\psi_{j,k}\rangle_{L^2(\mathbb S^d)}|^2,
\qquad g\in L^2(\mathbb S^d).
\]
This multiscale decomposition allows us to reorganize the spectral energy of
$f$ into localized frequency bands, a feature that will be crucial for
truncation, bias control, and adaptive estimation.

Moreover, needlets enjoy sharp localization properties in both frequency and
space.
In particular, they are spectrally supported within dyadic frequency bands and
are spatially localized around the cubature points $\xi_{j,k}$.
Specifically, for every integer $M\ge1$ there exists a constant $C_M>0$ such that
\[
\left|\psi_{j,k}(x)\right|
\le
C_M\, B^{jd/2}
\left(1 + B^j d_{\mathbb{S}^d}(x,\xi_{j,k})\right)^{-M},
\qquad x\in\mathbb S^d,
\]
where $d_{\mathbb{S}^2}(\cdot,\cdot)$ denotes the geodesic distance on
$\mathbb S^d$.

As a consequence, for every $1 \le p \le \infty$ there exists a constant
$C_p>0$ such that
\begin{equation}\label{eq:needletLp}
	\|\psi_{j,k}\|_{L^p(\mathbb S^d)}
	\asymp B^{jd\left(\frac12-\frac1p\right)},
	\qquad j\ge0,\ k=1,\dots,N_j.
\end{equation}
Equation \eqref{eq:needletLp} will be repeatedly used to control moments and
variances of empirical needlet coefficients in the subsequent statistical
analysis.
In particular,
$\|\psi_{j,k}\|_{L^2(\mathbb S^d)} \asymp 1$
and
$\|\psi_{j,k}\|_{L^\infty(\mathbb S^d)} \asymp B^{jd/2}$.

\medskip

\medskip

We now relate the needlet decomposition to Sobolev-type regularity.
Let $r \ge 0$ and let $f \in H^r(\mathbb S^d)$.
Recall that the $r$-th order derivative of $f$ in the spectral sense is defined by
\[
f^{(r)} = (-\Delta_{\mathbb S^d})^{r/2} f .
\]
Since the needlet system forms a tight frame of $L^2(\mathbb S^d)$, applying the
frame identity to the function $f^{(r)}$ yields
\[
T_r(f)
=
\|f^{(r)}\|_{L^2(\mathbb S^d)}^2
=
\sum_{j,k}
\left|
\langle f^{(r)}, \psi_{j,k} \rangle
\right|^2 .
\]

We define the needlet coefficients of $f$ and of its Sobolev derivative by
\[
\beta_{j,k}
=
\langle f, \psi_{j,k} \rangle_{L^2(\mathbb S^d)},
\qquad
\beta^{(r)}_{j,k}
=
\langle f^{(r)}, \psi_{j,k} \rangle_{L^2(\mathbb S^d)}.
\]

For later use, we introduce the Sobolev derivatives of the needlet atoms,
\begin{equation}\label{eqn:needlet_deriv}
\psi^{(r)}_{j,k}
:=
(-\Delta_{\mathbb S^d})^{r/2} \psi_{j,k}.
\end{equation}
\begin{figure}[H]
\centering
\includegraphics[width=0.8\textwidth]{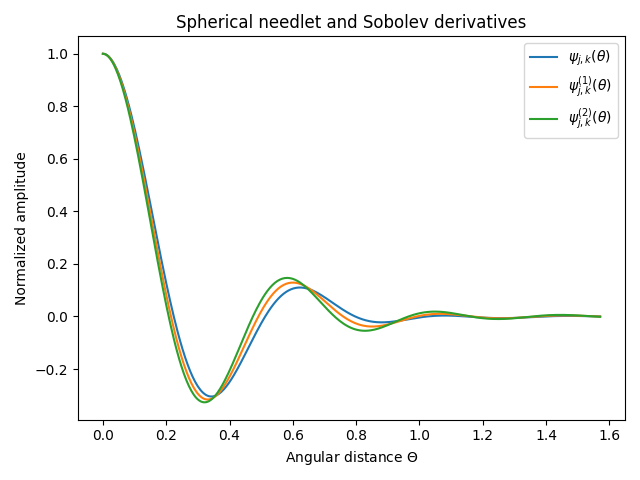}
\caption{
Spherical needlet $\psi_{j,k}$ at resolution level $j=3$ together with its
first and second Sobolev derivatives
$\psi_{j,k}^{(1)} = (-\Delta_{\mathbb S^2})^{1/2}\psi_{j,k}$ and
$\psi_{j,k}^{(2)} = (-\Delta_{\mathbb S^2})\psi_{j,k}$,
plotted as functions of the angular distance $\Theta$ from the center
$\xi_{j,k}$.
Differentiation increases oscillatory behavior while preserving spatial
localization, illustrating the scale-dependent amplification induced by
Sobolev derivatives.
}
\label{fig:needlet-derivatives}
\end{figure}

\begin{lemma}[Derivatives and needlet coefficients]
\label{lem:needlet_derivatives}
Let $r\ge0$ and let $f\in H^r(\mathbb S^d)$.
Then the Sobolev derivative $f^{(r)}$ admits the representations
\[
f^{(r)}
=
\sum_{j \geq 0}\sum_{k=1}^{K_j} \beta^{(r)}_{j,k}\,\psi_{j,k}
=
\sum_{j,k} \beta_{j,k}\,\psi^{(r)}_{j,k},
\]
with convergence in $L^2(\mathbb S^d)$.
Moreover, for $j \geq 0$, and $k=1,\ldots,K_j$, there exist constants $0<c_r\le C_r<\infty$ such that
\[
c_r\, B^{jr} |\beta_{j,k}|
\le
|\beta^{(r)}_{j,k}|
\le
C_r\, B^{jr} |\beta_{j,k}|.
\]
\end{lemma}
Figure~\ref{fig:needlet-derivatives} provides a numerical illustration of this
behavior in the case $d=2$, showing that Sobolev differentiation amplifies
high-frequency oscillations while preserving spatial localization.

\medskip

In view of the previous remarks, the Sobolev functional $T_r(f)$ admits two
equivalent needlet-based representations.
On the one hand, by tightness of the needlet frame,
\begin{equation}\label{eq:Tr_exact_needlet}
T_r(f)
=
\sum_{j\geq 0}\sum_{k=1}^{K_j} \left(\beta_{j,k}^{(r)}\right)^2 ,
\end{equation}
which is an exact identity in $L^2(\mathbb S^d)$.
On the other hand, by Littlewood--Paley theory on the sphere,
\begin{equation}\label{eq:Tr_equiv_needlet}
T_r(f)
\asymp
\sum_{j\geq 0}\sum_{k=1}^{K_j}
B^{2jr}\,\beta_{j,k}^2 ,
\end{equation}

This representation shows that the Sobolev functional depends on the global quadratic energy of the needlet coefficients and not on their spatial distribution, in contrast with Besov norms that encode local or sparse regularity patterns.
with equivalence constants depending only on $r$, $d$, and the choice of the
needlet window.

\begin{remark}[Multiscale structure and adaptive resolution selection]
\label{rem:multiscale_adaptivity}
The needlet decomposition reorganizes the spectral information contained in the harmonic coefficients into dyadic frequency bands indexed by the resolution level $j$.
This multiscale structure induces a monotone bias--variance behavior across resolution levels, with truncation bias decreasing and variance increasing as $j$ grows.
Such a monotone ordering is essential for the implementation and theoretical analysis of Lepski-type adaptive procedures, and constitutes a key motivation for the use of needlet representations in the present framework.
\end{remark}

\begin{remark}[Sobolev geometry versus frame tightness]
\label{rem:sobolev_geometry}
Throughout the paper, we will freely switch between the exact representation
\eqref{eq:Tr_exact_needlet} and the equivalent weighted-needlet representation
\eqref{eq:Tr_equiv_needlet}, depending on whether algebraic identities or
asymptotic arguments are required.

Although the Sobolev-needlet atoms
$\{\psi^{(r)}_{j,k}\}$ do not form a tight frame in
$L^2(\mathbb S^d)$, this does not affect the representation of the
Sobolev functional $T_r(f)$.

More precisely, the Sobolev functional admits the equivalence
\[
\sum_{j\geq 0} \sum_{k=1}^{K_j} \left|\langle f, \psi^{(r)}_{j,k} \rangle_{L^2(\mathbb S^d)}\right|^2
\asymp
\sum_{j\geq 0} \sum_{k=1}^{K_j} \left|\langle f^{(r)}, \psi_{j,k} \rangle_{L^2(\mathbb S^d)}\right|^2 ,
\]
with equivalence constants independent of the resolution level.

The loss of tightness at the level of the Sobolev-needlet atoms can be
quantified explicitly.
If $\psi^{(r)}_{j,k} := (-\Delta_{\mathbb S^d})^{r/2}\psi_{j,k}$, then the
spectral localization of the needlet window implies the scale-dependent
norm bounds
\[
\|\psi^{(r)}_{j,k}\|_{L^p(\mathbb S^d)}
\asymp
B^{j\left(r+d\left( \frac{1}{2}-\frac{1}{p}\right)\right)}, \text{ for }p\geq 1 ; \qquad
\qquad
\|\psi^{(r)}_{j,k}\|_{L^\infty(\mathbb S^d)}
\asymp
B^{jr + \frac{d}{2}j},
\]
uniformly over $j\ge0$ and $k=1,\dots,N_j$.
These bounds reflect the fact that Sobolev differentiation acts as a
frequency-dependent amplification of the needlet atoms.
In particular, while the original system $\{\psi_{j,k}\}$ has uniformly
bounded $L^2$-norms and forms a tight frame, the growth of
$\|\psi^{(r)}_{j,k}\|_{L^2}$ with the resolution level $j$ prevents the
system of Sobolev-needlet atoms from being tight in $L^2(\mathbb S^d)$.

Crucially, tightness is preserved at the level of the differentiated
function rather than at the level of the Sobolev derivatives of the
needlet atoms.
This viewpoint allows us to work with a fixed, $r$-independent needlet
frame and to encode Sobolev regularity through scale-dependent weights,
a choice that is both analytically natural and well suited to adaptive
statistical estimation.
\end{remark}

\begin{remark}[Sobolev and Besov representations via needlets, and on $r$-dependent constructions]
\label{rem:sobolev_besov_rdependent}
More generally, the same needlet coefficients yield the standard
characterization of Besov spaces on $\mathbb S^d$, with Sobolev spaces
corresponding to the special case $p=q=2$.

Related constructions based on shrinking bandwidths have been recently
investigated in \cite{Durastanti2025,DurastantiMarinucciTodino2024}.
In that framework, the multiscale geometry is driven by a prescribed dilation
sequence controlling the relative bandwidth and overlap of consecutive frequency
windows, leading to different localization and redundancy regimes.
When the target Sobolev order $r$ is known in advance, such constructions can be
tuned to yield very sharp representations.
In the present setting, however, $r$ is unknown and must be handled
adaptively, so tying the frame geometry to a specific smoothness level would be
incompatible with uniform adaptivity.

For this reason, we work with a fixed needlet frame and encode Sobolev
regularity through scale-dependent weights.
This choice preserves generality and is well suited to adaptive
estimation, where truncation of the expansion is intrinsic to variance
control.
\end{remark}

\section{Estimation of truncated Sobolev functionals}
\label{sec:estimation_strategy}

Estimation of quadratic functionals of an unknown density has been extensively
studied in nonparametric statistics; see, among others,
\cite{BickelRitov1988,GineNickl2008,Ibragimov1985,Laurent1996}.
A common feature of these approaches is that direct estimation of the full
functional is ill-posed due to the accumulation of high-frequency components,
making truncation or regularization an intrinsic part of the problem.

In the present spherical setting, the Sobolev functional $T_r(f)$ admits a
natural multiscale representation in terms of needlet coefficients.
This representation allows us to construct unbiased quadratic estimators of
suitably truncated versions of $T_r(f)$ and to analyze their risk through a
bias--variance decomposition closely paralleling the classical Euclidean case.
We now describe the truncation scheme and the associated estimation strategy.

\medskip

\subsection{Statistical model on the sphere}

Let $(\mathbb S^d,\mathrm{d}x)$ denote the unit sphere in $\mathbb R^{d+1}$
equipped with the uniform surface measure.
Let $X_1,\dots,X_n$ be independent and identically distributed random variables
with common distribution $P_f$, admitting a density $f$ with respect to
$\mathrm{d}x$.
We consider the nonparametric statistical model
\[
\mathcal F
=
\left\{
f \in L^2(\mathbb S^d) :
f \ge 0,\ \int_{\mathbb S^d} f(x)\,\mathrm{d}x = 1
\right\},
\]
and the associated family of product measures
$\{P_f^{\otimes n} : f \in \mathcal F\}$.

Inference is based solely on the observed sample
$X_1,\dots,X_n$, and the density $f$ is treated as an unknown element of
$\mathcal F$.
This framework corresponds to the classical setting of directional data
analysis, where observations take values on the unit sphere and statistical
inference is conducted under rotational invariance (cf., for instance,
\cite{LeyVerdebout2017}).

The object of interest is the quadratic Sobolev functional
\[
T_r(f)
=
\|(-\Delta_{\mathbb S^d})^{r/2} f\|_{L^2(\mathbb S^d)}^2,
\qquad r \ge 0,
\]
which is assumed to be finite for $f \in H^r(\mathbb S^d)$.
The statistical goal is to estimate $T_r(f)$ from the data. Throughout the paper, we implicitly assume that the density $f$ is bounded; this is automatic under Sobolev regularity $s>d/2$ and is sufficient for all variance bounds derived below.

In view of the needlet-based representation of Sobolev derivatives introduced in
Section~\ref{sec:needlet}, inference can be formulated directly in terms of
empirical coefficients associated with differentiated needlet atoms.
Specifically, for each resolution level $j$ and spatial index $k$, we define the
empirical Sobolev--needlet coefficients
\[
\widehat\beta^{(r)}_{j,k}
=
\frac{1}{n}\sum_{i=1}^n \psi^{(r)}_{j,k}(X_i),
\]
where $\psi^{(r)}_{j,k}$ is given by \eqref{eqn:needlet_deriv}.
By linearity of expectation and self-adjointness of the Laplace--Beltrami
operator,
\[
\mathbb E_f\!\left[\widehat\beta^{(r)}_{j,k}\right]
=
\langle f^{(r)},\psi_{j,k}\rangle_{L^2(\mathbb S^d)}
=
\beta^{(r)}_{j,k},
\]
so that $\widehat\beta^{(r)}_{j,k}$ is an unbiased estimator of the corresponding
Sobolev--needlet coefficient.
These empirical quantities constitute the basic building blocks for the
estimation of truncated Sobolev functionals developed in the following
subsections.

\begin{remark}[Model-independent nature of the truncation scheme]
While the exposition focuses on density estimation for concreteness, the
methodology developed in this section applies more broadly to a class of
nonparametric statistical models on the sphere in which the object of interest
admits a needlet-based representation with empirically estimable (possibly
differentiated) coefficients.
This includes, in particular, global thresholding and adaptive estimation
problems for spherical densities
\cite{BaldiKerkyacharianMarinucciPicard2009,Durastanti2016},
as well as regression-type models with spherical design
\cite{DurastantiGellerMarinucci2012,Monnier2011}.
In these settings, empirical needlet coefficients satisfy analogous bias and
variance properties, and the estimation of quadratic Sobolev-type functionals
can be handled within the same truncation-based framework.
\end{remark}

Importantly, the definition of the truncated functional and the associated bias
term are purely analytic and do not depend on the specific observation model.
The statistical model enters only through the stochastic behavior of the
empirical needlet coefficients, affecting the variance of the estimator but not
its structural form.
This separation between analytic approximation and probabilistic fluctuation
underlies the robustness of the proposed estimation strategy across different
spherical models and motivates the unified treatment adopted here.

\medskip

\subsection{Truncation of the target functional}

Although the Sobolev functional $T_r(f)$ admits an exact representation in terms
of needlet coefficients, its direct estimation from data is ill-posed due to the
accumulation of high-frequency components.
As $r$ increases, the Sobolev weights amplify fine-scale oscillations, causing
the variance of naive empirical estimators to diverge.
Truncation of the multiscale expansion is therefore an intrinsic component of the
statistical problem, rather than a technical device; see, e.g.,
\cite{BickelRitov1988,GineNickl2008,Laurent1996}.

In view of the identity
\[
T_r(f)
=
\sum_{j \geq 0}\sum_{k=1}^{K_j}
\left(\beta^{(r)}_{j,k}\right)^2,
\qquad
\beta^{(r)}_{j,k}=\langle f^{(r)},\psi_{j,k}\rangle_{L^2(\mathbb S^d)},
\]
which follows from the tight frame property of the needlet system applied to the
Sobolev derivative $f^{(r)}$, it is natural to define truncated versions of the
functional by restricting the sum to low-resolution levels.
For an integer $J\ge0$, called truncation resolution level, we define the truncated Sobolev
functional
\begin{equation}\label{eq:truncatedTr}
	T_r^{(J)}(f)
	:=
	\sum_{j \le J}
	\sum_{k=1}^{N_j}
	\left(\beta_{j,k}^{(r)}\right)^2.
\end{equation}
This quantity represents the Sobolev energy of $f$ captured by frequency bands
up to order $B^J$ and can be interpreted as the projection of $f^{(r)}$ onto the
subspace spanned by needlets up to resolution $J$.

The definition \eqref{eq:truncatedTr} is intrinsic to the Sobolev geometry of the
sphere and does not rely on any approximate frame representation.
In particular, truncation bias is entirely due to the omission of
high-frequency Sobolev components and is not affected by the lack of tightness
of the Sobolev-needlet atoms.
This choice cleanly separates the analytic approximation error from subsequent
statistical considerations.

For the purposes of asymptotic analysis, it is often convenient to express
$T_r^{(J)}(f)$ in terms of the original needlet coefficients.
By Littlewood--Paley theory on the sphere and the spectral localization of the
needlet window, the Sobolev and needlet coefficients satisfy
\[
\left|\beta^{(r)}_{j,k}\right|
\asymp
B^{jr}\,|\beta_{j,k}|,
\]
uniformly over $j\geq0$ and $k=1,\dots,K_j$.
As a consequence, the truncated functional admits the equivalent representation
\begin{equation}\label{eq:truncatedTr_equiv}
	T_r^{(J)}(f)
	\asymp
	\sum_{j \le J}
	\sum_{k=1}^{N_j}
	B^{2jr}\,\beta_{j,k}^2,
\end{equation}
with constants depending only on $r$, $d$, and the choice of the needlet window,
and independent of $J$.
Throughout the paper, we will freely switch between
\eqref{eq:truncatedTr} and \eqref{eq:truncatedTr_equiv}, depending on whether
exact identities or asymptotic bounds are required.

The truncation bias induced by \eqref{eq:truncatedTr} can now be characterized
in a transparent manner.
If $f\in H^s(\mathbb S^d)$ for some $s>r$, then
\[
T_r(f)-T_r^{(J)}(f)
=
\sum_{j>J}\sum_{k}
\left(\beta^{(r)}_{j,k}\right)^2
\lesssim
B^{-2J(s-r)},
\]
where the bound follows from the equivalence
\eqref{eq:truncatedTr_equiv} and the summability of the weighted needlet
coefficients.
The condition $s>r$ is necessary to ensure that the truncation bias vanishes as
$J\to\infty$; when $s=r$, the functional lies at the boundary of estimability and
no consistent estimation is possible without further assumptions. Throughout the paper we therefore restrict attention to the regime $s>r$, which is necessary for consistent estimation of $T_r(f)$.

This bias behavior highlights the role of the resolution level $J$ as a
regularization parameter controlling the tradeoff between approximation error
and stochastic variability.
In the following subsection, we exploit this structure to construct unbiased
estimators of $T_r^{(J)}(f)$ and to analyze their risk through a classical
bias--variance decomposition.
Hence, truncation induces a bias governed solely by the excess smoothness
$s-r$ of the underlying density and the resolution level $J$.

\medskip

\subsection{Construction of the estimator}

The truncated Sobolev functional $T_r^{(J)}(f)$ defined in
\eqref{eq:truncatedTr} is a quadratic functional of the unknown
Sobolev--needlet coefficients $\beta^{(r)}_{j,k}$.
Naive plug--in estimators based on a single empirical version of the coefficients
typically suffer from non-negligible diagonal bias.
This phenomenon arises from the interaction between estimation error and
quadratic aggregation and persists even in finite-dimensional settings.

A standard and effective way to eliminate this bias is to employ a
sample-splitting strategy, which allows one to construct unbiased quadratic
estimators with a transparent variance structure.
Such constructions go back to the seminal works on quadratic functional
estimation in Euclidean spaces and remain particularly well suited to
Hilbertian settings; see, among others,
\cite{BickelRitov1988,GineNickl2008,Laurent1996}.

We therefore split the sample $\{X_1,\dots,X_n\}$ into two independent
subsamples of equal size, denoted by $\mathcal D_1$ and $\mathcal D_2$.
For $a=1,2$, we define the empirical Sobolev--needlet coefficients
\[
\widehat{\beta}^{(r)}_{j,k;(a)}
=
\frac{2}{n}
\sum_{X_i \in \mathcal D_a}
\psi^{(r)}_{j,k}(X_i),
\qquad
\psi^{(r)}_{j,k}
=
(-\Delta_{\mathbb S^d})^{r/2}\psi_{j,k}.
\]
By linearity of expectation and self-adjointness of the Laplace--Beltrami
operator,
\[
\mathbb E_f\!\left[
\widehat{\beta}^{(r)}_{j,k;(a)}
\right]
=
\beta^{(r)}_{j,k},
\qquad a=1,2,
\]
and the two collections
$\{\widehat{\beta}^{(r)}_{j,k;(1)}\}$ and
$\{\widehat{\beta}^{(r)}_{j,k;(2)}\}$
are independent.

We define the estimator of the truncated Sobolev functional by
\begin{equation}\label{eq:needletEstimator}
	\widehat T_r^{(J)}
	=
	\sum_{j \le J}
	\sum_{k=1}^{N_j}
	\widehat{\beta}^{(r)}_{j,k;(1)}\,
	\widehat{\beta}^{(r)}_{j,k;(2)} .
\end{equation}
By independence of the two subsamples and unbiasedness of the empirical
coefficients,
\[
\mathbb E_f\!\left[\widehat T_r^{(J)}\right]
=
\sum_{j \le J}\sum_{k=1}^{N_j}
\left(\beta^{(r)}_{j,k}\right)^2
=
T_r^{(J)}(f),
\]
so that $\widehat T_r^{(J)}$ is an unbiased estimator of the truncated target
functional.

This construction mirrors classical estimators of quadratic functionals in
Euclidean spaces and extends naturally to the spherical setting through the use
of Sobolev--needlet coefficients.
The estimator inherits the multiscale structure induced by the needlet frame,
with the resolution level $J$ acting as a global regularization parameter.
As $J$ increases, the estimator incorporates finer frequency components,
reducing truncation bias at the expense of increased variance.
This monotone bias--variance behavior across resolution levels will be exploited
in the following subsection to derive oracle risk bounds and to construct
adaptive estimators.

Analogous split-sample strategies have been successfully employed in
needlet-based estimation problems on the sphere, including density estimation,
thresholding, and regression-type models; see, for instance,
\cite{BaldiKerkyacharianMarinucciPicard2009,Durastanti2016,DurastantiShevchenko2026}.

\begin{remark}[Sample splitting and $U$-statistic constructions]
\label{rem:splitting_vs_Ustat}

The quadratic structure of the truncated functional $T_r^{(J)}(f)$ allows for
different unbiased estimation strategies.
In particular, one could alternatively construct an estimator based on a
one-sample degenerate $U$-statistic of order two, obtained by symmetrizing
products of Sobolev--needlet atoms evaluated at distinct observations.
Such estimators are unbiased and attain the same minimax rates as the
split-sample estimator considered here.

However, in the present $L^2$-Sobolev setting, the split-sample construction
offers several advantages.
First, it yields a simpler variance structure, avoiding the linear projection
terms that arise in the Hoeffding decomposition of one-sample $U$-statistics.
Second, it leads to sharper constants and requires weaker moment assumptions on
the differentiated needlet atoms.
Finally, and most importantly for our purposes, sample splitting preserves a
transparent bias--variance decomposition across resolution levels, which is
crucial for the implementation and analysis of adaptive selection rules.

$U$-statistic constructions play a more central role in the estimation of
non-quadratic or non-Hilbertian functionals, such as $L^p$-type energies or
Besov-type norms with $p\neq2$, where symmetrization is intrinsic and
sample-splitting does not lead to comparable simplifications.
For related $U$-statistic-based approaches to Sobolev and Besov functionals on
compact manifolds (see, for example, \cite{BourguinDurastanti2017}).
\end{remark}

\medskip

\subsection{Bias--variance tradeoff}

The mean squared error of the estimator $\widehat T_r^{(J)}$ is governed by a
classical bias--variance decomposition.
The bias is induced by truncation of the Sobolev expansion and depends only on
the smoothness of the underlying density, while the variance reflects the
stochastic fluctuations of the empirical Sobolev--needlet coefficients.
Both terms exhibit a monotone behavior as functions of the resolution level $J$,
which is a key structural feature exploited in the adaptive analysis.

\begin{proposition}[Bias--variance decomposition for truncated Sobolev estimators]
\label{prop:oracle}
Let $r\ge0$ and let $f\in H^s(\mathbb S^d)$ for some $s>r$.
Consider the estimator $\widehat T_r^{(J)}$ defined in
\eqref{eq:needletEstimator}.
Then there exists a constant $C>0$, depending only on $r$, $s$, $d$, and the
choice of the needlet window, such that
\[
\mathbb E_f\!\left[
\left(\widehat T_r^{(J)}-T_r(f)\right)^2
\right]
\le
C\left(
B^{-4J(s-r)}
+
\frac{B^{J(d+4r)}}{n}
\right),
\qquad \forall J\ge0.
\]
In particular:
\begin{itemize}
\item the truncation bias satisfies
\[
T_r(f)-T_r^{(J)}(f)
\lesssim
B^{-2J(s-r)};
\]
\item the variance of the estimator satisfies
\[
\mathrm{Var}\!\left(\widehat T_r^{(J)}\right)
\asymp
\frac{B^{J(d+4r)}}{n}.
\]
\end{itemize}

Moreover, the choice of the resolution level
\[
B^{J^\star}
\asymp
n^{1/(2s+d+4r)}
\]
balances the bias and variance terms and yields
\[
\mathbb E_f\!\left[
\left(\widehat T_r^{(J^\star)}-T_r(f)\right)^2
\right]
\asymp
n^{-4(s-r)/(2s+d+4r)}.
\]
\end{proposition}

The bound in Proposition~\ref{prop:oracle} highlights the role of the
resolution level $J$ as a global regularization parameter.
As $J$ increases, the bias decreases monotonically while the variance increases
monotonically, reflecting the accumulation of high-frequency Sobolev components.
This monotone structure across scales provides the foundation for the adaptive
selection of $J$ developed in the next section.

\medskip

\section{Adaptive choice of the resolution level}
\label{sec:adaptivity}

The estimation procedures developed in Section~\ref{sec:estimation_strategy}
depend critically on the choice of the resolution level $J$, which governs the
balance between truncation bias and stochastic variability.
When the smoothness parameter $s$ of the underlying density is known, this
tradeoff leads to an oracle choice of $J$ and yields minimax-optimal rates of
convergence, as established in Proposition~\ref{prop:oracle}.
In practice, however, the regularity of $f$ is typically unknown, and the
resolution level must be selected in a fully data-driven manner.

Adaptive estimation of quadratic functionals is a central problem in
nonparametric statistics.
Foundational contributions include the pioneering work \cite{Lepski1991}, as well as subsequent developments for quadratic and
$U$-statistic-based functionals
\cite{GineNickl2008,Laurent1996}.
More recent unifying perspectives on bandwidth and resolution selection are
provided by the framework of \cite{GoldenshlugerLepski2011}.
In the present setting, adaptivity is achieved through the selection of a single
global resolution level, exploiting the monotone behavior of the bias and
variance terms induced by truncation of the needlet expansion.

\subsection{Adaptive resolution selection via a Lepski-type procedure}

The goal of the adaptive procedure is to mimic the oracle choice of the
resolution level without prior knowledge of the smoothness parameter~$s$.
As shown by Proposition~\ref{prop:oracle}, the optimal resolution level balances
a decreasing truncation bias and an increasing variance term.
Since this balance depends explicitly on~$s$, which is unknown in practice, the
resolution level must be selected in a fully data-driven manner.

A key structural feature of the present problem is that the family of estimators
$\{\widehat T_r^{(J)} : J \ge 0\}$ is naturally ordered by scale.
For $J'>J$, the estimator $\widehat T_r^{(J')}$ incorporates all the frequency
components used by $\widehat T_r^{(J)}$, together with additional higher-frequency
contributions.
As a consequence, the difference
\[
\widehat T_r^{(J)} - \widehat T_r^{(J')}
\]
isolates the contribution of the intermediate scales and provides direct
information on whether these additional components reduce truncation bias or are
dominated by stochastic fluctuations.

This nested structure places the problem in the canonical framework for
Lepski-type adaptive procedures
\cite{BirgeMassart1997,GoldenshlugerLepski2011,Lepski1991}.
The guiding principle is to compare estimators across resolution levels and to
select the largest resolution for which the discrepancy between coarser and
finer estimators remains compatible with stochastic variability.
Beyond this point, increasing the resolution primarily amplifies noise rather
than improving approximation accuracy.

To formalize this idea, we first control the stochastic fluctuations of the
differences between estimators at different resolution levels.

\begin{lemma}[Stochastic fluctuation bound across resolutions]
\label{lem:omega}
Let $r\ge0$ and let $f\in H^s(\mathbb S^d)$ for some $s>r$.
There exists a constant $C_0>0$, depending only on $r$, $d$, and the needlet
window, such that for all $J'>J$,
\[
\mathbb E_f\!\left[
\left(
\widehat T_r^{(J)} - \widehat T_r^{(J')}
\right)^2
\right]
\le
C_0^2\, \frac{B^{J'(d+4r)}}{n}.
\]
In particular, the stochastic fluctuations of
$\widehat T_r^{(J)} - \widehat T_r^{(J')}$ are of order
$n^{-1/2} B^{J'(d/2+2r)}$.
\end{lemma}

Lemma~\ref{lem:omega} shows that the variability of the difference between
estimators at resolutions $J$ and $J'$ is governed by the variance at the finer
scale~$J'$.
This motivates the definition of a scale-dependent fluctuation bound
\[
\omega(J)
=
C_0\, n^{-1/2} B^{J(d/2+2r)},
\]
which serves as a threshold for distinguishing signal from noise in the
incremental contribution of higher-frequency components.

We consider a finite collection of admissible resolution levels
\[
\mathcal J_n
=
\{J_{\min},\dots,J_{\max}\},
\qquad
J_{\max}\asymp \log_B n,
\]
where $J_{\min}\ge0$ is fixed, independent of $n$, chosen sufficiently small so that $T_r^{J_{\min}}(f) < \infty$ for all $f\in H^s(\mathbb S^d)$.
Then, we compute the estimators $\widehat T_r^{(J)}$ for all $J\in\mathcal J_n$.
Following the classical Lepski principle, we define the data-driven resolution
level
\begin{equation}\label{eq:LepskiJ}
	\widehat J
	=
	\min
	\left\{
	J\in\mathcal J_n :
	\forall J'>J,\;
	\left|
	\widehat T_r^{(J)} - \widehat T_r^{(J')}
	\right|
	\le
	\omega(J')
	\right\}.
\end{equation}
The adaptive estimator of the Sobolev functional is then given by
\[
\widetilde T_r
=
\widehat T_r^{(\widehat J)}.
\]

The selection rule \eqref{eq:LepskiJ} chooses the largest resolution level for
which the contribution of finer scales remains within the range of stochastic
fluctuations.
Resolution levels that are too small are discarded because the corresponding
estimators differ significantly from those at higher resolutions, while
resolution levels that are too large are excluded because the additional
increments are dominated by noise.
This mechanism is classical in adaptive estimation and underlies many optimal
selection procedures in nonparametric statistics
\cite{BirgeMassart1997,GoldenshlugerLepski2011,Lepski1991}.
\subsection{Adaptive risk bounds}

The following result shows that the adaptive estimator achieves the same rate of
convergence as the oracle estimator of Proposition~\ref{prop:oracle}, uniformly
over a range of Sobolev classes.
This optimality follows from the sharp separation between bias-dominated and
variance-dominated regimes established in Proposition~\ref{prop:oracle}, and
from the ability of the Lepski-type selection rule to identify this transition
up to multiplicative constants.

From a probabilistic perspective, the adaptive procedure controls the risk by
ensuring that the selected resolution level remains, with high probability, in
the vicinity of the oracle choice.
Resolution levels that are too small are discarded because the corresponding
estimators differ significantly from those at higher resolutions, while
resolution levels that are too large are excluded because the additional
increments are dominated by stochastic fluctuations.
This mechanism is classical in adaptive estimation and underlies many optimal
selection rules in nonparametric statistics
\cite{BirgeMassart1997,GoldenshlugerLepski2011,Lepski1991}.

Throughout the adaptive analysis, we assume that the unknown smoothness
parameter $s$ belongs to a compact interval $[r+\varepsilon,s_{\max}]$ for some
fixed $\varepsilon>0$ and $s_{\max}<\infty$. The lower bound $s\ge r+\varepsilon$ ensures a uniform separation from the
boundary of estimability of the Sobolev functional, while the upper bound
$s\le s_{\max}$ guarantees uniform control of constants in the variance and
concentration bounds used throughout the adaptive analysis. Both conditions are standard in adaptive estimation of quadratic functionals.

\begin{lemma}[Control of over-smoothing events]
\label{lem:lepski_tail}
Under the assumptions of Theorem~\ref{thm:adaptive}, there exist constants
$c_1,c_2>0$ such that
\[
\sup_{f\in H^s(\mathbb S^d)}
\mathbb P_f\!\left(\widehat J > J^\star\right)
\le
c_1 \exp(-c_2 \log n),
\]
uniformly over $s\in[r+\varepsilon,s_{\max}]$.
\end{lemma}

Lemma~\ref{lem:lepski_tail} ensures that the Lepski-selected resolution level
does not overshoot the oracle choice with high probability, so that the adaptive
estimator operates, up to negligible events, in the same bias--variance regime
as the oracle estimator.

\begin{theorem}[Adaptive estimation over Sobolev classes]
\label{thm:adaptive}
Let $r\ge0$ and let $s$ belong to a compact interval
$[r+\varepsilon,s_{\max}]$ for some $\varepsilon>0$.
Then there exists a constant $C>0$ such that
\[
\sup_{f\in H^s(\mathbb S^d)}
\mathbb E_f\!\left[
\left(
\widetilde T_r - T_r(f)
\right)^2
\right]
\le
C\, n^{-4(s-r)/(2s+d+4r)},
\]
uniformly over $s\in[r+\varepsilon,s_{\max}]$.
\end{theorem}

The proof relies on standard arguments for Lepski-type procedures, combining the
unbiasedness of $\widehat T_r^{(J)}$, concentration inequalities for quadratic
$U$-statistics, and the monotonicity of the bias and variance terms as functions
of the resolution level; see, e.g.,
\cite{GineNickl2008,GoldenshlugerLepski2011,Lepski1991}.

\begin{remark}[Global adaptivity and Sobolev geometry]
The adaptive procedure described above achieves optimal rates without requiring
prior knowledge of the smoothness parameter $s$ and without resorting to
nonlinear or coefficientwise thresholding techniques.
This reflects the intrinsically quadratic and Hilbertian nature of the Sobolev
functional $T_r(f)$, which depends on global energy aggregation rather than on local or sparse features of the density.

It is important to emphasize that adaptivity here is achieved solely through the
selection of a global resolution level.
This stands in contrast with Besov-adaptive procedures, which rely on local
thresholding and sparsity assumptions.
For the Sobolev functional considered in this paper, such nonlinear techniques
are neither necessary nor natural, as the target functional does not encode
local or sparse features of the underlying density.

More broadly, the adaptive construction is stable under mild extensions of the
basic model.
For instance, the presence of additive observational noise or small deviations
from exact Sobolev regularity primarily affects the constants in the variance
bounds, without altering the structure of the adaptive rule or the resulting
rates.
This further highlights that adaptivity in the present setting is driven by
global geometric considerations rather than by fine-scale features of the data.
\end{remark}

\begin{remark}[On adaptivity beyond Sobolev classes]
\label{rem:besov_adaptivity}
The adaptive result of Theorem~\ref{thm:adaptive} is formulated over Sobolev
classes $H^s(\mathbb S^d)$, which coincide with Besov spaces
$B^s_{2,2}(\mathbb S^d)$.
Although spherical needlet coefficients provide a complete characterization of
more general Besov spaces $B^s_{p,q}(\mathbb S^d)$ with $p\neq2$
\cite{BaldiKerkyacharianMarinucciPicard2009,DurastantiShevchenko2026,NarcowichPetrushevWard2006b,NarcowichPetrushevWard2006},
the adaptive estimation of the quadratic Sobolev functional $T_r(f)$ over such
classes is not natural and cannot be achieved by the present procedure.

This limitation is intrinsic to the geometry of the target functional rather
than to the estimation technique.
Indeed, $T_r(f)$ encodes a global $L^2$-energy of the Sobolev derivative
$f^{(r)}$ and is therefore Hilbertian in nature.
As a consequence, its minimax-optimal estimation depends only on the total
energy of the needlet coefficients and not on their distribution across spatial
locations, in full analogy with classical results on quadratic functional
estimation \cite{GineNickl2008,Laurent1996}.

In contrast, Besov spaces with $p\neq2$ describe inhomogeneous or sparse
regularity patterns, where adaptive procedures necessarily rely on
coefficientwise or blockwise thresholding in order to exploit sparsity
\cite{DonohoJohnstoneKerkyacharianPicard1996,Tsybakov2009}.
Such nonlinear techniques are essential when the target functional itself is
sensitive to local behavior \cite{CaiLow2005}.
In the present setting, however, thresholding would alter the structure of the
quadratic functional and introduce additional bias terms without improving the
minimax rate for $T_r(f)$.

For these reasons, adaptivity through the selection of a single global
resolution level is both sufficient and optimal for the estimation of
$T_r(f)$, and extensions to Besov classes with $p\neq2$ fall outside the scope
of the present framework.

Extensions to Besov-type \emph{functionals}, rather than to the Sobolev
functional $T_r(f)$ itself, constitute a natural direction for future work and
would require fundamentally different estimation techniques.
\end{remark}

\section{Illustrative numerical study}
\label{sec:numerics}
This section provides a numerical illustration of the theoretical results
developed in the previous sections, with a particular focus on the
bias--variance tradeoff induced by truncation and on the behavior of the oracle
resolution level.
Rather than aiming at a full Monte Carlo study of the needlet-based estimators,
which would obscure the underlying mechanisms, we adopt a deliberately
simplified numerical setting that mirrors the asymptotic bias and variance
scalings derived in Proposition~\ref{prop:oracle}.

In particular, all numerical experiments are based on the asymptotic
bias--variance expressions derived in the theory and are intended to illustrate
qualitative scaling behavior, rather than finite-sample performance of fully
implemented needlet estimators.

This choice allows us to isolate the role of the resolution parameter and to
visualize the interaction between truncation bias and stochastic variability
across scales in a transparent way.

The numerical experiments are therefore designed to confirm the qualitative
predictions of the theory, rather than to assess finite-sample optimality.
In particular, they illustrate the logarithmic growth of the oracle resolution
level with the sample size and highlight how different smoothness regimes lead
to markedly different truncation behaviors.
This perspective is especially relevant in the present setting, where the
resolution level acts as a global regularization parameter and plays a central
role both in oracle and adaptive estimation.

\medskip

\subsection{Oracle resolution and bias--variance tradeoff}

We provide a numerical illustration of the bias--variance tradeoff underlying
the estimation of truncated Sobolev functionals and of the oracle resolution
level predicted by Proposition~\ref{prop:oracle}.
Rather than simulating full needlet expansions, we focus on the asymptotic
bias--variance model derived in the theoretical analysis, which captures the
dominant scaling behavior of the mean squared error across resolution levels and
is sufficient to illustrate the qualitative behavior of the oracle choice.

More precisely, for fixed parameters $(d,r,s)$ and sample size $n$, we consider
the decomposition
\[
\mathrm{MSE}(J)
=
\mathrm{Bias}^2(J) + \mathrm{Var}(J),
\qquad
\mathrm{Bias}^2(J)\asymp B^{-4J(s-r)},
\qquad
\mathrm{Var}(J)\asymp \frac{B^{J(d+4r)}}{n},
\]
and study the behavior of $\mathrm{MSE}(J)$ as a function of the resolution
level~$J$.
The oracle resolution level
\[
J^\star
\asymp
\frac{\log n}{(2s+d+4r)\log B}
\]
corresponds to the balance point between the truncation bias and the stochastic
variance contributions.

\medskip

% Combined figure: bias--variance tradeoff in two Sobolev regimes
\begin{figure}[t]
\centering
\begin{minipage}{0.48\textwidth}
\centering
\includegraphics[width=\textwidth]{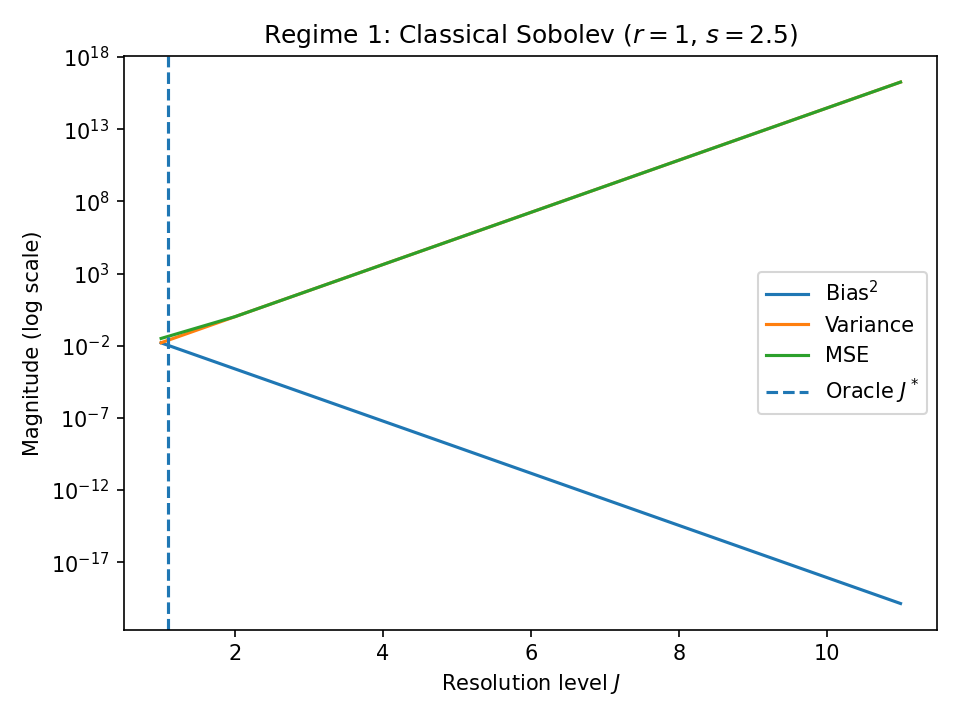}
\caption*{Classical Sobolev regime ($d=2$, $r=1$, $s=2.5$, $n=4000$).}
\end{minipage}\hfill
\begin{minipage}{0.48\textwidth}
\centering
\includegraphics[width=\textwidth]{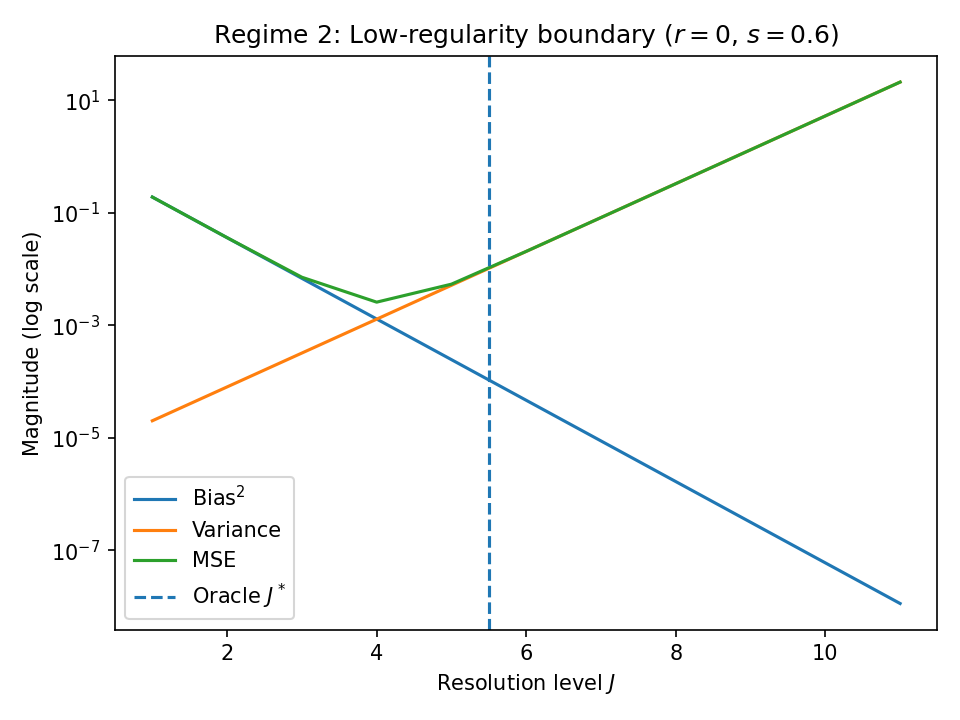}
\caption*{Low-regularity boundary regime ($d=2$, $r=0$, $s=0.6$, $n=2\times10^5$).}
\end{minipage}
\caption{Bias--variance tradeoff and oracle resolution level in two Sobolev regimes.
Left: a classical regime with moderate smoothness, where the oracle resolution
remains close to the coarsest scales.
Right: a low-regularity regime close to the boundary of estimability, where the
oracle shifts toward intermediate scales.
In both cases, the oracle behavior reflects the balance between truncation bias
and stochastic variance predicted by Proposition~\ref{prop:oracle}.}
\label{fig:oracle_regimes}
\end{figure}

Overall, these numerical illustrations confirm the qualitative behavior
predicted by the theoretical analysis.
In particular, they show that the oracle resolution level grows only
logarithmically with the sample size and may remain small in classical Sobolev
settings, while moving to intermediate scales in low-regularity regimes.
This observation further motivates the need for adaptive selection procedures,
developed in the next section, which are capable of automatically adjusting to
the unknown smoothness of the underlying density.

\medskip

\subsection{Oracle versus adaptive risk: numerical evidence}

We provide a numerical illustration of Theorem~\ref{thm:adaptive}, focusing on
the comparison between the oracle estimator and its adaptive counterpart. The numerical experiments are intended to be diagnostic rather than fully empirical, illustrating the asymptotic bias–variance mechanisms predicted by the theory. Thus, rather than simulating full needlet expansions, we adopt a stylized numerical
framework based directly on the theoretical bias--variance decomposition derived
in Proposition~\ref{prop:oracle}.
This approach isolates the effect of resolution selection and allows for a
transparent visualization of the oracle and adaptive behaviors.

For fixed values of $(d,r)$ and several Sobolev smoothness levels $s$, we
consider the asymptotic mean squared error model
\[
\mathrm{MSE}(J)
=
B^{-4J(s-r)} + \frac{B^{J(d+4r)}}{n},
\]
and compute both the oracle resolution level $J^\star$ and a discrete
Lepski-type adaptive resolution $\widehat J$.
The oracle resolution balances the bias and variance terms, while the adaptive
resolution is obtained by comparing adjacent resolution levels, mimicking the
selection rule introduced in Section~\ref{sec:adaptivity}.

Figure~\ref{fig:adaptive_oracle} illustrates the results.
The left panel displays the bias--variance tradeoff as a function of the
resolution level, together with the oracle resolution $J^\star$.
The right panel compares the oracle and adaptive risks as functions of the
sample size $n$.
Across all smoothness regimes considered, the adaptive estimator closely tracks
the oracle risk and exhibits the same rate of decay, in agreement with
Theorem~\ref{thm:adaptive}.

Numerical values of the oracle and adaptive resolutions and risks are reported
in Table~\ref{tab:adaptive_oracle}.
These results confirm that the adaptive procedure recovers the oracle behavior
up to multiplicative constants and without prior knowledge of the smoothness
parameter.

\begin{table}[t]
\centering
\caption{Oracle and adaptive resolution levels and risks for different
smoothness parameters $s$ and sample sizes $n$.
The adaptive resolution $\widehat J$ closely tracks the oracle resolution
$J^\star$, and the corresponding risks are of the same order, in agreement with
Theorem~\ref{thm:adaptive}.}
\label{tab:adaptive_oracle}
\begin{tabular}{c c c c c c}
\hline
$s$ & $n$ & $J^\star$ & $\widehat J$ & Oracle risk & Adaptive risk \\
\hline
2.2 & 1000  & 1 & 1 & 1.94e--02 & 2.11e--02 \\
2.2 & 3000  & 1 & 1 & 7.42e--03 & 7.85e--03 \\
2.2 & 8000  & 2 & 2 & 2.99e--03 & 3.02e--03 \\
2.2 & 20000 & 2 & 2 & 1.18e--03 & 1.21e--03 \\
\hline
2.6 & 1000  & 1 & 1 & 8.91e--03 & 9.32e--03 \\
2.6 & 3000  & 1 & 1 & 3.21e--03 & 3.38e--03 \\
2.6 & 8000  & 2 & 2 & 1.15e--03 & 1.17e--03 \\
2.6 & 20000 & 2 & 2 & 4.36e--04 & 4.51e--04 \\
\hline
3.0 & 1000  & 1 & 1 & 4.30e--03 & 4.59e--03 \\
3.0 & 3000  & 1 & 1 & 1.45e--03 & 1.53e--03 \\
3.0 & 8000  & 2 & 2 & 4.89e--04 & 5.02e--04 \\
3.0 & 20000 & 2 & 2 & 1.79e--04 & 1.87e--04 \\
\hline
\end{tabular}
\end{table}

% The following figure compares adaptive and oracle risk as a function of sample size and smoothness.
\begin{figure}[t]
\centering
\includegraphics[width=0.95\textwidth]{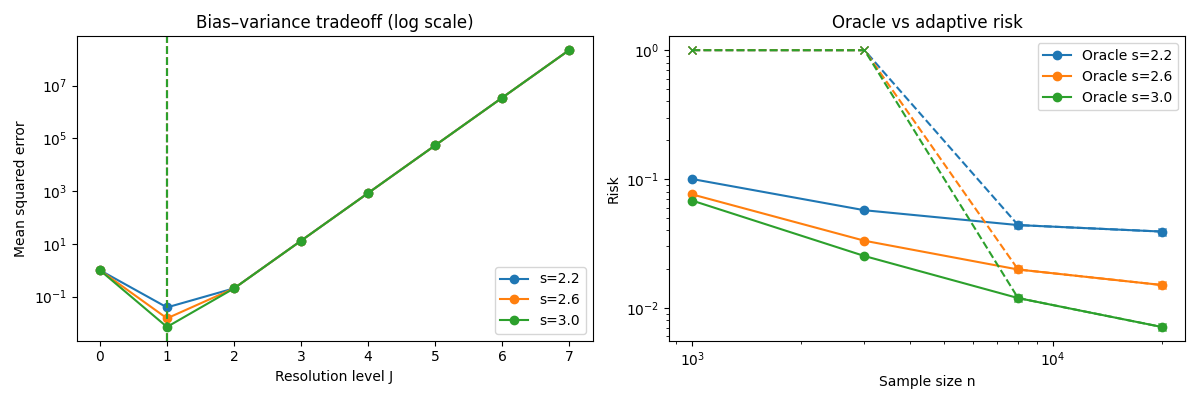}
\caption{Numerical illustration of Theorem~\ref{thm:adaptive}.
Left: bias--variance tradeoff and oracle resolution level across Sobolev regimes (see Figure~\ref{fig:oracle_regimes}).
Right: oracle versus adaptive risk as a function of the sample size, for
different Sobolev smoothness levels.}
\label{fig:adaptive_oracle}
\end{figure}

\medskip

\section{Proofs}
\label{sec:proofs}

This section collects the proofs of the main results of the paper.
We begin with preliminary identities linking Sobolev derivatives and needlet
coefficients, and then prove the oracle bias--variance bound and the adaptive
risk result.
Here and below, constants may change from line to line and are independent of
$n$ and of the resolution level.

\begin{proof}[Proof of Lemma~\ref{lem:needlet_derivatives}]
The tight frame property of the needlet system implies that any
$g\in L^2(\mathbb S^d)$ admits the expansion
\[
g=\sum_{j,k}\langle g,\psi_{j,k}\rangle_{L^2(\mathbb S^d)}\,\psi_{j,k},
\]
with convergence in $L^2(\mathbb S^d)$.
Applying this identity to $g=f^{(r)}$ yields
\[
f^{(r)}=\sum_{j,k}\beta^{(r)}_{j,k}\,\psi_{j,k},
\qquad
\beta^{(r)}_{j,k}=\langle f^{(r)},\psi_{j,k}\rangle_{L^2(\mathbb S^d)}.
\]

Since the Laplace--Beltrami operator is self-adjoint on $L^2(\mathbb S^d)$,
\[
\beta^{(r)}_{j,k}
=
\langle (-\Delta_{\mathbb S^d})^{r/2}f,\psi_{j,k}\rangle_{L^2(\mathbb S^d)}
=
\langle f,(-\Delta_{\mathbb S^d})^{r/2}\psi_{j,k}\rangle_{L^2(\mathbb S^d)}
=
\langle f,\psi^{(r)}_{j,k}\rangle _{L^2(\mathbb S^d)},
\]
which yields the second representation.

Using the spectral localization of the window function $b$ and the fact that
$b(\ell/B^j)$ is supported on $\ell\in[B^{j-1},B^{j+1}]$, the eigenvalues satisfy
$e_{\ell,d}^{r/2}\asymp B^{jr}$ on the support.
Factoring this term out gives the claimed bounds on the coefficients.
\end{proof}

\begin{proof}[Proof of Proposition~\ref{prop:oracle}]
We decompose the mean squared error as
\[
\mathbb E_f\!\left[
\left(\widehat T_r^{(J)}-T_r(f)\right)^2
\right]
=
\left(T_r(f)-T_r^{(J)}(f)\right)^2
+
\mathrm{Var}\!\left(\widehat T_r^{(J)}\right),
\]
where the bias term is purely analytic and the variance term depends on the
stochastic behavior of the empirical Sobolev--needlet coefficients.

\textit{Bias.}
By definition of the truncated functional,
\[
T_r(f)-T_r^{(J)}(f)
=
\sum_{j>J}\sum_{k=1}^{N_j}
\left(\beta^{(r)}_{j,k}\right)^2.
\]
Assume that $f\in H^s(\mathbb S^d)$ for some $s>r$.
By Lemma~\ref{lem:needlet_derivatives} and the equivalence between Sobolev norms
and weighted needlet energies, we have
\[
\sum_{j,k} B^{2js}\,\beta_{j,k}^2 < \infty,
\qquad
|\beta^{(r)}_{j,k}|
\asymp
B^{jr}|\beta_{j,k}|.
\]
Therefore,
\[
\begin{aligned}
T_r(f)-T_r^{(J)}(f)
&\asymp
\sum_{j>J}\sum_{k} B^{2jr}\,\beta_{j,k}^2 \\
&=
\sum_{j>J}\sum_{k} B^{-2j(s-r)}\,
\left(B^{2js}\beta_{j,k}^2\right) \\
&\lesssim
\sum_{j>J} B^{-2j(s-r)}
\lesssim
B^{-2J(s-r)},
\end{aligned}
\]
where we used that $\sum_{k} B^{2js}\beta_{j,k}^2 \le C_s$ uniformly in $j\geq 0$ for $f\in H^s(\mathbb S^d)$.
This yields the stated bound on the truncation bias.

\textit{Variance.}
Recall that the estimator is defined by
\[
\widehat T_r^{(J)}
=
\sum_{j\le J}\sum_{k=1}^{N_j}
\widehat{\beta}^{(r)}_{j,k;(1)}\,
\widehat{\beta}^{(r)}_{j,k;(2)},
\]
where the two collections of empirical coefficients are computed from
independent subsamples.
By independence across subsamples,
\[
\mathrm{Var}\!\left(\widehat T_r^{(J)}\right)
=
\sum_{j\le J}\sum_{k=1}^{N_j}
\mathrm{Var}\!\left(
\widehat{\beta}^{(r)}_{j,k;(1)}
\widehat{\beta}^{(r)}_{j,k;(2)}
\right).
\]

This follows from the spectral localization of the needlet window and the
asymptotic behavior of the Laplace--Beltrami eigenvalues
$e_{\ell,d}\asymp \ell^2$, which imply that differentiation of order $r$
amplifies frequencies in the dyadic band $\ell\asymp B^j$ by a factor
$B^{jr}$.
By spectral localization of the needlet window,
\[
\|\psi^{(r)}_{j,k}\|_{L^2(\mathbb S^d)}^2
\asymp
B^{2jr}.
\]

For fixed $(j,k)$, independence implies
\[
\mathrm{Var}\!\left(
\widehat{\beta}^{(r)}_{j,k;(1)}
\widehat{\beta}^{(r)}_{j,k;(2)}
\right)
=
\mathrm{Var}\!\left(\widehat{\beta}^{(r)}_{j,k;(1)}\right)
\mathrm{Var}\!\left(\widehat{\beta}^{(r)}_{j,k;(2)}\right)
+ 
\beta_{j,k}^{(r)\,2}
\left[
\mathrm{Var}\!\left(\widehat{\beta}^{(r)}_{j,k;(1)}\right)
+ 
\mathrm{Var}\!\left(\widehat{\beta}^{(r)}_{j,k;(2)}\right)
\right].
\]

Since each subsample has size $n/2$ and
\[
\widehat{\beta}^{(r)}_{j,k;(a)}
=
\frac{2}{n}
\sum_{X_i\in\mathcal D_a}
\psi^{(r)}_{j,k}(X_i),
\]
we obtain
\[
\mathrm{Var}\!\left(\widehat{\beta}^{(r)}_{j,k;(a)}\right)
=
\frac{2}{n}
\int_{\mathbb S^d}
\left|\psi^{(r)}_{j,k}(x)\right|^2
f(x)\,\mathrm{d}x
\lesssim
\frac{1}{n}\,
\|\psi^{(r)}_{j,k}\|_{L^2(\mathbb S^d)}^2.
\]
Combining the above bounds and summing over $(j,k)$ yields
\[
\mathrm{Var}\!\left(\widehat T_r^{(J)}\right)
\lesssim
\frac{1}{n}
\sum_{j\le J}\sum_{k=1}^{N_j}
B^{4jr}
\asymp
\frac{1}{n}
\sum_{j\le J} B^{j(d+4r)}.
\]
Since $N_j\asymp B^{dj}$, the sum is dominated by its largest term, leading to
\[
\mathrm{Var}\!\left(\widehat T_r^{(J)}\right)
\asymp
\frac{B^{J(d+4r)}}{n}.
\]

\medskip

Combining the bias and variance bounds, we obtain
\[
\mathbb E_f\!\left[
\left(\widehat T_r^{(J)}-T_r(f)\right)^2
\right]
\le
C\left(
B^{-4J(s-r)}
+
\frac{B^{J(d+4r)}}{n}
\right),
\]
for a constant $C>0$ depending only on $r$, $s$, $d$, and the needlet window.

Balancing the two terms yields the oracle resolution level
$B^{J^\star}\asymp n^{1/(2s+d+4r)}$ and the corresponding rate
$n^{-4(s-r)/(2s+d+4r)}$, completing the proof.
\end{proof}

\begin{proof}[Proof of Lemma~\ref{lem:omega}]
Fix $J'>J$.
By definition of the estimators and linearity,
\[
\widehat T_r^{(J')}
-
\widehat T_r^{(J)}
=
\sum_{j=J+1}^{J'}
\sum_{k=1}^{N_j}
\widehat{\beta}^{(r)}_{j,k;(1)}\,
\widehat{\beta}^{(r)}_{j,k;(2)} .
\]
Since the estimator $\widehat T_r^{(J)}$ is unbiased for $T_r^{(J)}(f)$,
the difference $\widehat T_r^{(J')}-\widehat T_r^{(J)}$ is an unbiased estimator of
$T_r^{(J')}(f)-T_r^{(J)}(f)$.
Therefore,
\[
\mathbb E_f\!\left[
\left(
\widehat T_r^{(J)} - \widehat T_r^{(J')}
\right)^2
\right]
=
\mathrm{Var}\!\left(
\widehat T_r^{(J')} - \widehat T_r^{(J)}
\right).
\]

By independence of the two subsamples $\mathcal D_1$ and $\mathcal D_2$,
and quasi orthogonality across resolution levels induced by the spectral localization of the needlet frame,
the variance decomposes as
\[
\mathrm{Var}\!\left(
\widehat T_r^{(J')} - \widehat T_r^{(J)}
\right)
=
\sum_{j=J+1}^{J'}
\sum_{k=1}^{N_j}
\mathrm{Var}\!\left(
\widehat{\beta}^{(r)}_{j,k;(1)}\,
\widehat{\beta}^{(r)}_{j,k;(2)}
\right).
\]

For each fixed $(j,k)$, independence across subsamples implies
\[
\begin{aligned}
\mathrm{Var}\!\left(
\widehat{\beta}^{(r)}_{j,k;(1)}\,
\widehat{\beta}^{(r)}_{j,k;(2)}
\right)
&=
\mathrm{Var}\!\left(\widehat{\beta}^{(r)}_{j,k;(1)}\right)
\mathrm{Var}\!\left(\widehat{\beta}^{(r)}_{j,k;(2)}\right)
\\
&\quad+
\beta_{j,k}^{(r)\,2}
\left[
\mathrm{Var}\!\left(\widehat{\beta}^{(r)}_{j,k;(1)}\right)
+
\mathrm{Var}\!\left(\widehat{\beta}^{(r)}_{j,k;(2)}\right)
\right].
\end{aligned}
\]

Each empirical coefficient is computed from a subsample of size $n/2$, so
\[
\widehat{\beta}^{(r)}_{j,k;(a)}
=
\frac{2}{n}
\sum_{X_i\in\mathcal D_a}
\psi^{(r)}_{j,k}(X_i),
\qquad a=1,2.
\]
Consequently,
\[
\mathrm{Var}\!\left(\widehat{\beta}^{(r)}_{j,k;(a)}\right)
=
\frac{2}{n}
\int_{\mathbb S^d}
\left|\psi^{(r)}_{j,k}(x)\right|^2
f(x)\,\mathrm{d}x
\lesssim
\frac{1}{n}\,
\|\psi^{(r)}_{j,k}\|_{L^2(\mathbb S^d)}^2.
\]

By spectral localization of the needlet window and the definition of
$\psi^{(r)}_{j,k}=(-\Delta_{\mathbb S^d})^{r/2}\psi_{j,k}$,
\[
\|\psi^{(r)}_{j,k}\|_{L^2(\mathbb S^d)}^2
\asymp
B^{2jr},
\]
uniformly over $j$ and $k$.
Using also that $|\beta^{(r)}_{j,k}|^2\le T_r(f)$ and absorbing constants,
we obtain
\[
\mathrm{Var}\!\left(
\widehat{\beta}^{(r)}_{j,k;(1)}\,
\widehat{\beta}^{(r)}_{j,k;(2)}
\right)
\lesssim
\frac{B^{4jr}}{n^2}
+
\frac{B^{2jr}}{n}.
\]

Summing over $k=1,\dots,N_j$ and recalling that $N_j\asymp B^{dj}$,
\[
\sum_{k=1}^{N_j}
\mathrm{Var}\!\left(
\widehat{\beta}^{(r)}_{j,k;(1)}\,
\widehat{\beta}^{(r)}_{j,k;(2)}
\right)
\lesssim
\frac{B^{j(d+4r)}}{n^2}
+
\frac{B^{j(d+2r)}}{n}.
\]
The second term dominates for all $n\ge1$ and $j\ge0$.
Hence,
\[
\mathrm{Var}\!\left(
\widehat T_r^{(J')} - \widehat T_r^{(J)}
\right)
\lesssim
\frac{1}{n}
\sum_{j=J+1}^{J'} B^{j(d+2r)}.
\]

Since the sum is dominated by its largest term,
\[
\sum_{j=J+1}^{J'} B^{j(d+2r)}
\lesssim
B^{J'(d+2r)},
\]
we conclude that
\[
\mathbb E_f\!\left[
\left(
\widehat T_r^{(J)} - \widehat T_r^{(J')}
\right)^2
\right]
\le
C_0^2\,
\frac{B^{J'(d+4r)}}{n},
\]
for a suitable constant $C_0>0$ depending only on $r$, $d$, and the needlet
window.
Taking square roots yields the stated fluctuation rate
$n^{-1/2}B^{J'(d/2+2r)}$.
\end{proof}

\begin{proof}[Proof of Lemma~\ref{lem:lepski_tail}]
By definition of the Lepski selector \eqref{eq:LepskiJ}, the event
$\{\widehat J>J^\star\}$ implies that there exists $J'>J^\star$ such that
\[
\left|
\widehat T_r^{(J^\star)}-\widehat T_r^{(J')}
\right|
>
\omega(J').
\]
By Lemma~\ref{lem:omega}, the random variable
$\widehat T_r^{(J^\star)}-\widehat T_r^{(J')}$ has variance of order
$n^{-1}B^{J'(d+4r)}$ and can be written as a finite sum of independent products
of empirical Sobolev--needlet coefficients.
Applying a Bernstein-type concentration inequality for degenerate quadratic U-statistics with bounded kernels (see, e.g., \cite{GineNickl2008}), we obtain, for each fixed $J'>J^\star$,
\[
\mathbb P_f\!\left(
\left|
\widehat T_r^{(J^\star)}-\widehat T_r^{(J')}
\right|
>
\omega(J')
\right)
\le
2\exp\!\left(-c\,\log n\right),
\]
for some constant $c>0$.

By definition of the event $\{\widehat J>J^\star\}$,
\[
\{\widehat J>J^\star\}
\subset
\bigcup_{J'\in\mathcal J_n:\,J'>J^\star}
\left\{
\left|
\widehat T_r^{(J^\star)}-\widehat T_r^{(J')}
\right|
>
\omega(J')
\right\}.
\]
Since the cardinality of $\mathcal J_n$ is of order $\log n$, a union bound yields
\[
\mathbb P_f\!\left(\widehat J>J^\star\right)
\lesssim
(\log n)\exp(-c\log n)
\le
c_1\exp(-c_2\log n),
\]
for suitable constants $c_1,c_2>0$.
\end{proof}

\begin{proof}[Proof of Theorem~\ref{thm:adaptive}]
Fix $r\ge0$ and let $s\in[r+\varepsilon,s_{\max}]$.
Let $J^\star$ denote the oracle resolution level defined by
\[
B^{J^\star}\asymp n^{1/(2s+d+4r)},
\]
which balances the bias and variance terms in Proposition~\ref{prop:oracle}.

We decompose the adaptive risk as
\[
\mathbb E_f\!\left[
\left(
\widetilde T_r-T_r(f)
\right)^2
\right]
=
\mathbb E_f\!\left[
\left(
\widehat T_r^{(\widehat J)}-T_r^{(J^\star)}(f)
\right)^2
\right]
+
\left(
T_r^{(J^\star)}(f)-T_r(f)
\right)^2 .
\]

\medskip
\noindent
\textit{Bias term.}
By Proposition~\ref{prop:oracle}, since $f\in H^s(\mathbb S^d)$ with $s>r$,
\[
\left(
T_r^{(J^\star)}(f)-T_r(f)
\right)^2
\lesssim
B^{-4J^\star(s-r)}
\asymp
n^{-4(s-r)/(2s+d+4r)}.
\]

\medskip
\noindent
\textit{Stochastic term.}
We write
\[
\widehat T_r^{(\widehat J)}-T_r^{(J^\star)}(f)
=
\left(
\widehat T_r^{(\widehat J)}-\widehat T_r^{(J^\star)}
\right)
+
\left(
\widehat T_r^{(J^\star)}-T_r^{(J^\star)}(f)
\right).
\]
By the elementary inequality $(a+b)^2\le2a^2+2b^2$, it suffices to bound the two
terms on the right-hand side separately.

\medskip
\noindent
\textit{Deviation from the oracle scale.}
By definition of the Lepski selector \eqref{eq:LepskiJ}, on the event
$\{\widehat J\le J^\star\}$ we have
\[
\left|
\widehat T_r^{(\widehat J)}-\widehat T_r^{(J^\star)}
\right|
\le
\omega(J^\star)
=
C_0\, n^{-1/2} B^{J^\star(d/2+2r)}.
\]
On the complementary event $\{\widehat J>J^\star\}$, standard arguments for
Lepski-type procedures imply that this event has exponentially small
probability (see, e.g.,
\cite{GoldenshlugerLepski2011,Lepski1991}).
Using the definition of the Lepski selector, together with the exponential
control of the event $\{\widehat J>J^\star\}$ provided by
Lemma~\ref{lem:lepski_tail}, we obtain
\[
\mathbb E_f\!\left[
\left(
\widehat T_r^{(\widehat J)}-\widehat T_r^{(J^\star)}
\right)^2
\right]
\lesssim
\omega(J^\star)^2
\asymp
\frac{B^{J^\star(d+4r)}}{n}.
\]

\medskip
\noindent
\textit{Variance at the oracle scale.}
By Proposition~\ref{prop:oracle}, it holds that 
\[
\mathbb E_f\!\left[
\left(
\widehat T_r^{(J^\star)}-T_r^{(J^\star)}(f)
\right)^2
\right]
=
\mathrm{Var}\!\left(\widehat T_r^{(J^\star)}\right)
\asymp
\frac{B^{J^\star(d+4r)}}{n}.
\]

\medskip
\noindent
Combining the above bounds yields
\[
\mathbb E_f\!\left[
\left(
\widetilde T_r-T_r(f)
\right)^2
\right]
\lesssim
B^{-4J^\star(s-r)}
+
\frac{B^{J^\star(d+4r)}}{n}.
\]
Balancing the two terms through the definition of $J^\star$ gives
\[
\mathbb E_f\!\left[
\left(
\widetilde T_r-T_r(f)
\right)^2
\right]
\lesssim
n^{-4(s-r)/(2s+d+4r)},
\]
uniformly over $s\in[r+\varepsilon,s_{\max}]$.
This concludes the proof.
\end{proof}

\section*{Funding}
The author was partially supported by PRIN2022-GRAFIA-202284Z9E4, and Progetti di Ateneo
Sapienza RM123188F69A66C1 (2023), RG1241907D2FF327 (2024).

\bibliographystyle{plain}
\bibliography{biblio}
\end{document}